\documentclass{amsart}
\usepackage{amssymb,amsmath,amsthm,amsfonts}
\def\R {\mathbb{R}}
\def\C {\mathcal{C}}
\def\N {\mathbb{N}}
\def\S {\mathbb{S}}

\def\from {\colon}
\def\bv {\mathbf{v}}
\def\bu {\mathbf{u}}
\def\bw {\mathbf{w}}
\def\bW {\mathbf{W}}
\def\eps{\varepsilon}
\def\dist{{\rm dist}}

\newcommand{\loc}{\mathrm{loc}}

\renewcommand{\div}{\mathrm{div}}

\newcommand{\problem}[1] {(P)_{#1}}

\newcommand{\nuACF} {\nu^{\mathrm{ACF} }}

\newcommand{\de}[1] {\mathrm{d} #1}

\newcommand{\ddfrac}[2] {\frac{\displaystyle #1 }{\displaystyle #2} }

\DeclareMathOperator*{\pv}{pv}
\DeclareMathOperator*{\tsum}{\textstyle{\sum}}

\DeclareMathOperator{\supp}{supp}

\newtheorem{proposition}{Proposition}[section]
\newtheorem{theorem}[proposition]{Theorem}
\newtheorem{corollary}[proposition]{Corollary}
\newtheorem{lemma}[proposition]{Lemma}

\theoremstyle{definition}
\newtheorem{definition}[proposition]{Definition}
\newtheorem{remark}[proposition]{Remark}
\numberwithin{equation}{section}

\title[Segregation problems involving s-laplacians]{Uniform H\"older regularity with small exponent
in competition-fractional diffusion systems.}
\author{Susanna Terracini}
\email{susanna.terracini@unito.it}
\address{Dipartimento di Matematica "Giuseppe Peano", Universit\`a degli Studi
di Torino, Via Carlo Alberto 10, 10123 Torino, Italy}

\author{Gianmaria Verzini }
\email{gianmaria.verzini@polimi.it}
\address{Dipartimento di Matematica, Politecnico di Milano, p.za Leonardo da
Vinci 32,  20133 Milano, Italy}

\author{Alessandro Zilio}
\email{alessandro.zilio@mail.polimi.it}
\address{Dipartimento di Matematica, Politecnico di Milano, p.za Leonardo da
Vinci 32,  20133 Milano, Italy}

\thanks{Work partially supported by the PRIN2009 grant ``Critical Point Theory and
Perturbative Methods for Nonlinear Differential Equations''.}

\begin{document}

\maketitle

\begin{abstract}
For a class of competition-diffusion nonlinear systems involving the $s$-power of the Laplacian,
$s\in(0,1)$, of the form
\[
(-\Delta)^{s} u_i=f_i(u_i) - \beta u_i\sum_{j\neq i}a_{ij}u_j^2,\qquad i=1,\dots,k,
\]
we prove that $L^\infty$ boundedness implies $\C^{0,\alpha}$ boundedness for $\alpha>0$
sufficiently small, uniformly as $\beta\to +\infty$. This extends to the case $s\neq1/2$
part of the results obtained by the authors in the previous paper [arXiv:1211.6087v1].
\end{abstract}

\section{Introduction}

In this paper we study the problem
\begin{equation}\label{eqn: general_system}
\begin{cases}
(-\Delta)^s u_i=f_{i,\beta}(u_i)-\beta u_i\tsum_{j\neq i}a_{ij}u_j^2 \\
u_i\in H^{s}(\R^N),
\end{cases}
\end{equation}
in dimension $N\geq2$, where $a_{ij}=a_{ji}>0$, $\beta$ is positive and large, and the non-local operator
\[
(-\Delta)^s u (x) = c_{N,s} \pv\int_{\R^N} \frac{u(x)-u(\xi)}{|x-\xi|^{N+2s}}\,\de\xi
\]
denotes the $s$-power of the laplacian. We are mostly concerned with the asymptotic behavior of the solutions to the previous system as the parameter $\beta \to +\infty$: as we shall see, this entails spatial segregation for the limiting profiles. Our aim is to prove uniform in $\beta$ bounds in H\"older spaces, extending to the case $s\in(0,1)$ part of the results that we already obtained for the case $s=1/2$ in the recent paper \cite{tvz1}, to which we refer for further details.

Segregation-diffusion problems arise in different applicative contests, from biological models
for competing species to the phase-segregation phenomenon in Bose-Einstein condensation. Regarding
the standard diffusion case ($s=1$), a broad literature is present. Among the others, we mention the papers
\cite{ctvOptimal,ctv,MR2278412,MR2393430,ww,MR2529504,wz,dwz1,tt,dwz2,dwz3}, which are mostly concerned with
regularity issues. Our study is motivated by the recent interest that has developed around equations involving fractional
laplacians, as they model long-jump diffusion processes in population dynamics, and they naturally appear in relativistic corrections of quantum field theory.

Exploiting the local realization of the fractional laplacian $(-\Delta)^s$ as a Dirichlet-to-Neumann map (see, for instance, \cite{cs}),
semilinear problems involving fractional laplacians have been the object of a massive study. Accordingly, letting $a:=1-2s \in (-1,1)$, if we introduce the differential operator (on the $(N+1)$-dimensional space)
\[
    L_a v := -\div\left(|y|^a \nabla v\right),
\]
and define
\[
    \partial^a_{\nu} v := \lim_{y \to 0^+} - y^a \partial_{y} v,
\]
we obtain that, up to normalization constants, the problem
\[
    \begin{cases}
    L_a  v_i = 0 & \text{in } B^+_1\\
    \partial^a_{\nu} v_i = f_{i,\beta}(v_i) - \beta v_i \tsum_{j \neq i} a_{ij}v_j^2 & \text{on } \partial^0 B^+_1
    \end{cases} \leqno \problem{\beta}^s
\]
is a localized version of \eqref{eqn: general_system}, with $u_i(x)=v_i(x,0)$. Here, as usual, we write
$\R^{N+1}_+\ni X = (x,y)$ and $B^+_r(x_0,0):= B_r(x_0,0)\cap\{y>0\}$, which boundary contains the spherical part $\partial^+B^+_r :=\partial B_r \cap\{y>0\}$ and the flat one $\partial^0B^+_r :=B_r \cap\{y=0\}$. Well known properties of the Muckenhoupt $A_2$-weights (see for instance \cite{kufner}) allow to provide a weak formulation of $\problem{\beta}^s$ in the weighted space
\[
    H^{1;a}(\Omega) := \left\{v : \int_{\Omega} y^a \left(|v|^2 + |\nabla v|^2 \right)\, \de{x}\de{y} < \infty \right\},
\]
endowed with its natural Hilbert structure.

The main result we prove in this paper is the following.
\begin{theorem}[Local uniform H\"older bounds]\label{thm:_local_holder}
Let the functions $f_{i,\beta}$ be continuous and uniformly bounded (w.r.t. $\beta$) on bounded sets.
There exists $\alpha = \alpha(N,s)>0$ such that, for every $\{\bv_{\beta}\}_{\beta }$ family of $H^{1;a}(B^+_1)$ solutions to the problems $\problem{\beta}^s$,
\[
    \| \bv_{\beta} \|_{L^{\infty}(B^+_1)} \leq M
    \quad\implies\quad
    \| \bv_\beta\|_{\C^{0,\alpha}\left(\overline{B^+_{1/2}}\right)} \leq C,
\]
where $C=C(M,\alpha)$. Furthermore, $\{\bv_\beta\}_{\beta > 0}$ is relatively compact in $H^{1;a}(B^+_{1/2}) \cap \C^{0,\alpha}\left(\overline{B^+_{1/2}}\right)$.
\end{theorem}
The above result allows to prove its natural global counterpart, either on the whole of $\R^N$ or on
domains with suitable boundary conditions.
\begin{theorem}[Global uniform H\"older bounds]\label{thm: glob_bdd_intro}
Let $f_{i,\beta}$ and $\alpha$ be as in the previous theorem, and let $\{\bu_{\beta}\}_{\beta }$ be a family of $H^{s}(\R^N)$ solutions to the problems
\[
    \begin{cases}
    ( -\Delta)^{s} u_i= f_{i,\beta}(u_i) - \beta u_i \tsum_{j \neq i} a_{ij}u_j^2 & \text{in } \;\Omega\\
    u_i\equiv 0 &\text{in} \; \R^N\setminus \Omega,
    \end{cases}
\]
where $\Omega$ is a bounded domain of $\R^N$, with smooth boundary. Then
\[
    \| \bu_{\beta} \|_{L^{\infty}(\Omega)} \leq M
    \quad\implies\quad
    \| \bu_\beta\|_{\C^{0,\alpha}(\R^N)} \leq C(M,\alpha).
\]
\end{theorem}
Of course, a natural question regards the optimal regularity of such problems, that is
the maximal value of $\alpha$ for the above results to hold. In the case of the standard
diffusion ($s=1$), the analogous issue is faced in \cite{nttv}, where uniform H\"older bounds
are shown for every $\alpha<1$. The proof of this result relies on a blow-up procedure,
leading to a contradiction with some Liouville type theorems; these are based on the validity
of some monotonicity formulae of Alt-Caffarelli-Friedman and Almgren type.
In \cite{tvz1}, we consider the case $s=1/2$. In the situation there, a
two-step strategy has been developed: indeed, though providing some uniform H\"older bounds, the above blow-up procedure  seems not enough to catch the optimal regularity threshold.
The main reason for this failure is the lack of an exact Alt-Caffarelli-Friedman formula,
so that the bounds, at a first stage, are obtained only when $\alpha$ is smaller than some number $\nuACF>0$,
which is not explicit.
Nonetheless, this provides enough compactness to trigger the second step of the strategy,
based on the classification of the possible profiles obtained through a blow-down argument.
At the end of the procedure, uniform bounds for any $\alpha<1/2$ are shown.
In this perspective, Theorem \ref{thm:_local_holder} here corresponds to the first step (the blow-up procedure) of the
strategy just described, extended to the general case $s\in(0,1)$. The exponent $\alpha$
mentioned there is subject to two main restrictions: as before,
$\alpha$ is bounded above by the minimal rate of growth for multi-phase segregation profiles $\nuACF$;
on the other hand, when $s>1/2$, a new upper threshold must be taken into account, which is
related to the phenomenon of self-segregation.

The first restriction, as we mentioned, is related to the validity of an exact Alt-Caffarelli-Friedmann
formula, which in turn depends on an optimal partition problem. More precisely, let $\S^{N}_+ := \partial^+ B^+$.
For each open $\omega\subset\S^{N-1}:=\partial\S^N_+$ we define the first $s$-eigenvalue
associated to $\omega$ as
\begin{equation}\label{eqn: deflambda1}
    \lambda_1^s(\omega) :=
    \inf\left\{
    \frac{\int_{\S^{N}_+} y^a|\nabla_{T} u|^2 \, \de{\sigma}
    }{\int_{\S^{N}_+} y^a u^2\, \de{\sigma}} :
    u \in H^{1;a}(\S^{N}_+), \, u\equiv0 \text{ on }\S^{N-1}\setminus\omega \right\},
\end{equation}
where $\nabla_{T} u$ is the tangential gradient of $u$ on $\S^{N}_+$.
The minimal rate of growth for multi-phase segregation profiles is given by the number
\begin{equation}\label{eqn: def_nuacf}
\begin{split}
    \nuACF :&= \inf\left\{\frac{\gamma({\lambda_1^s(\omega_1)})+\gamma({\lambda_1^s(\omega_2)})}2 :{\omega_1\cap\omega_2=\emptyset}\right\},
\end{split}
\end{equation}
where, as usual,
\[
\gamma(t) := \sqrt{ \left(\frac{N-2s}{2} \right)^2 + t} - \frac{N-2s}{2}
\]
is defined in such a way that $u$ achieves $\lambda_1^s(\omega)$ if and only if it is one signed, and its $ \gamma (\lambda_1^s(\omega))$-homogeneous extension to $\R^{N+1}_+$ is $L_a$-harmonic. As
a peculiar difference with respect to the case $s=1$, we remark that
the eigenfunctions achieving $\nuACF$ have not disjoint support
on the whole $\S^N_+$, but only on its boundary $\S^{N-1}$. In particular, the
degenerate partition $(\emptyset,\S^{N-1})$ is admissible, and one can show that
it has the same level than the equatorial cut one:
\[
\frac{\gamma({\lambda_1^s(\emptyset)})+\gamma({\lambda_1^s(\S^{N-1})})}2
=
\frac{\gamma({\lambda_1^s(\S^{N-1}_+)})+\gamma({\lambda_1^s(\S^{N-1}_-)})}2
=s.
\]
As a consequence, the above optimal partition problem does not enjoy the same convexity
properties than the one corresponding to $s=1$, and we can only show that
\[
0<\nuACF\leq s.
\]

Turning to self-segregation, the main point is that the fundamental solution
\[
\Gamma(X) = \frac{C_{N,s}}{|X|^{N-2s}},
\]
turns out to be bounded near $0$ and $H^{1;a}(B)$, whenever $s>1/2$, $N=1$. This
implies that, when $s>1/2$, $N\geq2$, the function
\[
v(x,y) = (x_1^2+y^2)^{(2s-1)/2}
\]
is positive and $L_a$-harmonic for $y>0$, $\partial_\nu^a v (x,0)=0$ whenever
$v(x,0)\neq0$, and its trace on $\R^N$ has disconnected positivity regions. Moreover, such
self-segregated profile is globally H\"older continuous, of exponent $\alpha = 2s - 1$ which
is arbitrarily small as $s \to (1/2)^+$.
The phenomenon of self-segregation can be excluded in some situations, for instance when $s\leq1/2$ (for
capacitary reasons), or when suitable minimality conditions are imposed (as in \cite{crs}). Nonetheless,
in general it is hard to tackle: for the case $s=1$ it was excluded only recently, in \cite{dwz3}.

To conclude we stress that, by exploiting the compactness provided by Theorem \ref{thm:_local_holder},
the optimal regularity should arise from the classification of
suitable blow-down profiles. Also this point presents a number of new difficulties
with respect to the case $s=1/2$, and it will be the object of a forthcoming paper.

\section{Monotonicity formulae}\label{sec:acf}
This section is devoted to the introduction of some monotonicity formulae, which will provide suitable estimates in order to prove some Liouville type results. Our first aim is to prove monotonicity formulae of Alt-Caffarelli-Friedman type for the one phase problem: these will imply non existence results for $L_a$-harmonic functions under different assumptions on their growth at infinity and on the geometry of their null set.

Secondly, we will concentrate on systems of degenerate elliptic equations, providing monotonicity formulae of Alt-Caffarelli-Friedman type with two phases, and of Almgren type.
\subsection{One phase Alt-Caffarelli-Friedman formulae}
We first deal with $L_a$-harmonic functions (on $\R^{N+1}_+$) which vanish on the whole $\R^N$.
\begin{proposition}\label{prp: ACF y 2s}
Let $v \in H^{1;a}(B_R^+)$ be a continuous function such that
\begin{itemize}
   \item  $v(x,0) = 0$ for $x \in \R^N$;
   \item  for every non negative $\phi \in \C_0^{\infty}(B_R)$,
   \[
   \int\limits_{\R^{N+1}_+}(L_a v)v\phi \, \de x \de y + \int\limits_{\R^N} (\partial_\nu^a v)v\phi \,
   \de x = \int\limits_{\R^{N+1}_+} y^a \nabla v\cdot\nabla(v\phi) \, \de x \de y \leq0.
   \]
\end{itemize}
Then the function
\[
    \Phi(r) := \frac{1}{r^{4s}} \int\limits_{B_r^+} y^a
\frac{|\nabla v|^2}{|X|^{N-2s} } \,\de{x} \de{y}
\]
is monotone non decreasing in $r$ for $r \in (0,R)$.
\end{proposition}

\begin{remark}\label{rem:chang_sign_ACF_segr}
Since
\begin{equation}\label{eqn: acf_per_moduli}
\int\limits_{\R^{N+1}_+}y^a \nabla v\cdot\nabla(v\phi) \, \de x \de y =
\int\limits_{\R^{N+1}_+}y^a \left[|\nabla v|^2\phi + \frac12\nabla v^2\cdot\nabla\phi\right] \de x \de y,
\end{equation}
we have that if $v$ satisfies the assumptions of Proposition \ref{prp: ACF y 2s} then
also $|v|$ does.
\end{remark}

\begin{definition}\label{def:Gamma_1}
We define ${\Gamma_1^s} \in \C^1(\R^{N+1}_+; \R^+)$ as
\[
    {\Gamma_1^s}(X) :=
    \begin{cases}
        \frac{1}{|X|^{N-2s}} & |X|\geq 1\\
        \frac{N+2(1-s)}{2} - \frac{N-2s}{2} |X|^2 & |X| < 1.
    \end{cases}
\]
We let also $\Gamma_{\eps}^s(X) = {\Gamma_1}^s(X/\eps) \eps^{2s-N}$, so that
$\Gamma_{\eps}^s \nearrow \Gamma^s = |X|^{2s-N}$, a multiple of the fundamental
solution of the $s$-laplacian, as $\eps \to 0$.
\end{definition}

\begin{remark}
We observe that each ${\Gamma_\eps^s}$ is radial and, in
particular, $\partial_{\nu}^a {\Gamma_\eps^s} = 0$ on $\R^N$. Moreover, since $N-2s>0$,
they are $L_a$-superharmonic on $\R^{N+1}_+$.
\end{remark}
The proof of Proposition \ref{prp: ACF y 2s} is based on the following calculation. Incidentally, we observe that
also the following monotonicity results rest on a similar argument.
\begin{lemma}\label{lem: phi is well defined}
Let $v$ be as in Proposition \ref{prp: ACF y 2s}. The
function
\begin{equation}\label{eqn: improper integral}
   r\mapsto\int\limits_{B_r^+} y^a \frac{|\nabla v |^2}{|X|^{N-2s}}
\,\de{x} \de{y}
\end{equation}
is well defined and bounded in any compact subset of $(0,1)$.
\end{lemma}
\begin{proof}
We proceed as follows: let $\eps > 0$, $\delta>0$ and let
$\eta_\delta\in \C^\infty_0 (B_{r+\delta})$ be a smooth, radial cutoff function such that
$0\leq\eta_\delta\leq1$ and $\eta_\delta=1$ on $B_r$.
Choosing $\phi=\eta_\delta \Gamma_{\eps}^s$ in the second assumption of Proposition \ref{prp: ACF y 2s},
and recalling equation \eqref{eqn: acf_per_moduli}, we obtain
\begin{multline*}
\int\limits_{\R^{N+1}_+}y^a \left[|\nabla v |^2\Gamma_{\eps}^s +
\frac12\nabla v^2\cdot\nabla\Gamma_{\eps}^s\right]\eta_\delta
\de x \de y
\leq -\int\limits_{\R^{N+1}_+}\frac12 y^a \Gamma_{\eps}^s\nabla v^2\cdot\nabla\eta_\delta \de x \de y\\
= \int\limits_{r}^{r+\delta} \left[-\eta_\delta'(\rho)\int\limits_{\partial^+B^+_\rho} y^a \Gamma_{\eps}^sv \nabla v
\cdot \frac{X}{|X|}\de\sigma \right]\de\rho.
\end{multline*}
Passing to the limit as $\delta\to0$ we obtain, for almost every $r\in(0,1)$,
\[
\int\limits_{B^+_r}y^a \left[|\nabla v |^2\Gamma_{\eps}^s +
\frac12\nabla (v )^2\cdot\nabla\Gamma_{\eps}^s\right]\de x \de y
\leq
\int\limits_{\partial^+B^+_r}y^a \Gamma_{\eps}^sv   \partial_\nu v \de\sigma,
\]
which, combined with the inequality $L_a \Gamma_{\eps}^s \geq 0$ tested with
$v ^2/2$ leads to
\[
    \int\limits_{B_r^+} y^a |\nabla v |^ 2 \Gamma_{\eps}^s \, \de{x}\de{y} \leq
\int\limits_{\partial^+ B_r^+}y^a \left(\Gamma_{\eps}^s v  {\partial_\nu v  }
 - \frac{v ^2}{2} {\partial_\nu} \Gamma_{\eps}^s
\right)\, \de{\sigma}.
\]
Letting $\eps \to 0^+$, by monotone convergence we infer
\begin{equation}\label{eqn: estim from above denom}
    \int\limits_{B_r^+} y^a \frac{|\nabla v |^ 2}{|X|^{N-2s}} \, \de{x}\de{y} \leq
\frac{1}{r^{N-2s}} \int\limits_{\partial^+ B_r^+} y^a v  \frac{\partial v
}{\partial \nu} \, \de{\sigma} + \frac{N-2s}{2r^{N+1-2s}} \int\limits_{\partial^+
B_r^+} y^a v ^2 \, \de{\sigma}
\end{equation}
and this, in turns, proves the lemma.
\end{proof}

\begin{proof}[Proof of Proposition \ref{prp: ACF y 2s}]
By Remark \ref{eqn: acf_per_moduli} we can assume, without loss of generality, that $v$ is
(non trivial and) non negative, and that $R=1$. We start observing that the function $\Phi(r)$ is
positive and absolutely continuous for $r \in (0,1)$. Therefore,
the proposition follows once we prove that $\Phi'(r) \geq 0$ for almost every $r \in
(0,1)$. A direct computation of the logarithmic derivative of $\Phi$ shows that
\[
    \frac{\Phi'(r)}{\Phi(r)} = -\frac{4s}{r} +
\dfrac{\int\limits_{\partial^+ B_r^+}y^a |\nabla v|^2 /|X|^{N-2s} \,\de{\sigma}
}{\int\limits_{B_r^+} y^a |\nabla v|^2 / |X|^{N-2s} \,\de{x} \de{y}}.
\]
First we use the estimate \eqref{eqn: estim from above denom} to bound from below the left hand side:
\begin{multline*}
    \dfrac{\int\limits_{\partial^+ B_r^+}y^a |\nabla v |^2/|X|^{N-2s}
\,\de{\sigma} }{\int\limits_{B_r^+} y^a |\nabla v |^2 /|X|^{N-2s} \,\de{x}
\de{y}} \geq  \dfrac{\int\limits_{\partial^+ B_r^+} y^a |\nabla v |^2 \,\de{\sigma}
}{\int\limits_{\partial^+ B_r^+} v  y^a \partial_{\nu} v  \, \de{\sigma} + (N-2s)
\frac{r}{2} \int\limits_{\partial^+ B_r^+} y^a v ^2 \, \de{\sigma} }\\ =  \frac{1}{r}
\dfrac{\int\limits_{\S^{N}_+} \xi_{N+1}^a |\nabla v ^{(r)}|^2 \,\de{\sigma} }{\int\limits_{\S^{N}_+}
v ^{(r)} \xi_{N+1}^a \partial_{\nu} v ^{(r)} \, \de{\sigma} + \frac{N-2s}{2}
\int\limits_{\S^{N}_+} \xi_{N+1}^a (v ^{(r)})^2 \, \de{\sigma} },
\end{multline*}
where $v ^{(r)}\from \S^{N-1}_+ \to \R$ is defined as $v ^{(r)}(\xi) =
v (r\xi)$, so that $y = r \xi_{N+1}$. We now
estimate the right hand side as follows: the numerator writes
\begin{multline*}
    \int\limits_{\S^{N}_+} \xi_{N+1}^a|\nabla v ^{(r)}|^2 \, \de{\sigma}  = \int\limits_{\S^{N}_+}
\xi_{N+1}^a |\partial_{\nu} v ^{(r)}|^2 \, \de{\sigma}  + \int\limits_{\S^{N}_+} \xi_{N+1}^a |\nabla_T
v ^{(r)}|^2 \, \de{\sigma}  \\
    = \int\limits_{\S^{N}_+} \xi_{N+1}^a |v ^{(r)}|^2 \, \de{\sigma}  \left( \underbrace{\frac{
\int\limits_{\S^{N}_+} \xi_{N+1}^a|\partial_{\nu} v ^{(r)}|^2 \, \de{\sigma} }{\int\limits_{\S^{N}_+}
\xi_{N+1}^a |v ^{(r)}|^2 \, \de{\sigma} } }_{t^2} + \underbrace{\frac{\int\limits_{\S^{N}_+}
\xi_{N+1}^a |\nabla_T v ^{(r)}|^2 \, \de{\sigma} }{\int\limits_{\S^{N}_+} \xi_{N+1}^a |v ^{(r)}|^2 \,
\de{\sigma}} }_{\mathcal{R}} \right).
\end{multline*}
where $\mathcal{R}$ stands for the Rayleigh quotient of $v ^{(r)}$ on
$\S^{N}_+$. On the other hand, by the Cauchy-Schwarz inequality, the denominator
may be estimated from above by
\begin{multline*}
    \int\limits_{\S^{N}_+} \xi_{N+1}^a v ^{(r)} \partial_{\nu} v ^{(r)} \, \de{\sigma} +
\frac{N-2s}{2} \int\limits_{\S^{N}_+} \xi_{N+1}^a |v ^{(r)}|^2 \, \de{\sigma} \\
\leq
\left(\int\limits_{\S^{N}_+} \xi_{N+1}^a |v ^{(r)}|^2 \, \de{\sigma}\right)^{1/2}
\left(\int\limits_{\S^{N}_+} \xi_{N+1}^a \partial_{\nu} v ^{(r)} \, \de{\sigma}\right)^{1/2} +
\frac{N-2s}{2} \int\limits_{\S^{N}_+} \xi_{N+1}^a |v ^{(r)}|^2 \, \de{\sigma}\\
    \leq\int\limits_{\S^{N}_+} \xi_{N+1}^a |v ^{(r)}|^2 \, \de{\sigma} \left[
\underbrace{\left(\frac{ \int\limits_{\S^{N}_+} \xi_{N+1}^a |\partial_{\nu} v ^{(r)}|^2 \,
\de{\sigma} }{\int\limits_{\S^{N}_+} \xi_{N+1}^a |v ^{(r)}|^2 \, \de{\sigma} }\right)^{1/2}}_{t} +
\frac{N-2s}{2} \right].
\end{multline*}
As a consequence
\begin{equation}\label{eqn: stima dal basso ACF}
    \dfrac{\int\limits_{\partial^+ B_r^+} y^a |\nabla v |^2/|X|^{N-2s}
\,\de{\sigma} }{\int\limits_{B_r^+} y^a |\nabla v |^2 /|X|^{N-2s} \,\de{x}
\de{y}} \geq \frac{1}{r} \min_{t \in \R^+} \frac{\mathcal{R}+t^2}{ t +
\frac{N-2s}{2}}.
\end{equation}
A simple computation shows that the minimum is achieved when
\[
    t = \gamma( {\mathcal{R}}) = \sqrt{ \left(\frac{N-2s}{2}\right)^2 +
\mathcal{R} } - \frac{N-2s}{2},
\]
and it is equal to $2\gamma( {\mathcal{R}})$.
Recalling the definition of $\lambda_1^s(\emptyset)$ (equation \eqref{eqn: deflambda1}) we obtain
\[
    \frac{\Phi'(r)}{\Phi(r)}  + \frac{4s}{r} \geq \frac{2}{r}\gamma(\lambda^s_1(\emptyset))
\]
and the proposition follows observing that $\lambda_1^s(\emptyset)$ is achieved by
$v(x,y)= y^{2s}$, in such a way that
\[
    \gamma\left(\lambda_1^s\left(\emptyset\right)\right) = 2s. \qedhere
\]
\end{proof}
Now we turn to functions which vanish only on a half space.
\begin{proposition}\label{prp: ACF sym}
Let $v \in H^{1;a}(B_R^+)$ be a continuous function such that
\begin{itemize}
   \item  $v(x,0) = 0$ for $x_1 \leq 0$;
   \item  for every non negative $\phi \in \C_0^{\infty}(B_R)$,
   \[
   \int\limits_{\R^{N+1}_+}(L_a v)v\phi \, \de x \de y + \int\limits_{\R^N} (\partial_\nu^a v)v\phi \,
   \de x = \int\limits_{\R^{N+1}_+} y^a \nabla v\cdot\nabla(v\phi) \, \de x \de y \leq0.
   \]
\end{itemize}
Then the function
\[
    \Phi(r) := \frac{1}{r^{2s}} \int\limits_{B_r^+} y^a
\frac{|\nabla v|^2}{|X|^{N-2s} } \,\de{x} \de{y}
\]
is monotone non decreasing in $r$ for $r \in (0,R)$.
\end{proposition}
\begin{proof}
The proof follows the line of the one of Proposition \ref{prp: ACF y 2s}, recalling that
\[
    v(x,y)= \left(\frac{\sqrt{x_1^2+y^2}+x_1}{2}\right)^s
\]
achieves
$
\gamma(\lambda_1^s(\S^{N-1}\cap\{x_1>0\}))=s
$
(see, for instance, \cite[page 442]{css}).
\end{proof}
In the previous propositions, we considered functions vanishing on the whole $\R^N$,
or on a half-space. Now, in great contrast with the case $s\leq1/2$, it is known that,
if $s>1/2$, then also $(N-1)$-dimensional subsets may have positive capacity. This motivates
the following formula, which is the analogous of the previous ones, for functions
which vanish on subspaces of $\R^N$ of codimension 1.
\begin{proposition}\label{prp: ACF N-1}
Let $s>1/2$ and let $v \in H^{1;a}(B_R^+)$ be a continuous function such that
\begin{itemize}
   \item  $v(x,0) = 0$ for $x_1 = 0$;
   \item  for every non negative $\phi \in \C_0^{\infty}(B_R)$,
   \[
   \int\limits_{\R^{N+1}_+}(L_a v)v\phi \, \de x \de y + \int\limits_{\R^N} (\partial_\nu^a v)v\phi \,
   \de x = \int\limits_{\R^{N+1}_+} y^a \nabla v\cdot\nabla(v\phi) \, \de x \de y \leq0.
   \]
\end{itemize}
Then the function
\[
    \Phi(r) := \frac{1}{r^{4s-2}} \int\limits_{B_r^+} y^a
\frac{|\nabla v|^2}{|X|^{N-2s} } \,\de{x} \de{y}
\]
is monotone non decreasing in $r$ for $r \in (0,R)$.
\end{proposition}
\begin{proof}
Let $\bar\omega=\S^{N-1}\setminus\{x_1=0\}$, and let us consider the function
\[
v(x,y)=|(x_1,0,y)|^{2s-1},
\]
that is the fundamental solution in dimension 1, extended in a constant way to the other directions.
Then $v$ is $(2s-1)$-homogeneous, positive and $L_a$-harmonic for $y>0$. We deduce that
its restriction to $\partial^+ B_1^+ = \S^N_+$ is an eigenfunction associated to
$\lambda_1^s(\bar\omega)$, so that
\[
\gamma(\lambda_1^s(\bar\omega))=2s-1.
\]
As a consequence, also in this case the proposition follows by reasoning as in the proof of
Proposition \ref{prp: ACF y 2s}.
\end{proof}

\subsection{Two phases Alt-Caffarelli-Friedman monotonicity formulae}
Now we turn to the multi-component ACF formulae. We start by proving that
the constant $\nuACF$ defined in equation \eqref{eqn: def_nuacf} is not 0.
\begin{lemma}\label{lem: nuacf>0}
For any $N \geq 2$, $0<\nuACF\leq s$.
\end{lemma}
\begin{proof}
The bound from above easily follows by comparing with the value corresponding
to the partition $(\S^{N-1},\emptyset)$: indeed, it holds $\lambda_1^s(\S^{N-1})=0$,
achieved by $u(x,y)\equiv1$, and $\lambda_1^s(\emptyset)=2sN$,
achieved by $u(x,y)=y^{1-a}$. In order to prove the estimate from below, one
can argue by contradiction, as in the proof of \cite[Lemma 2.5]{tvz1}, exploiting the compactness
both of the embedding $H^{1;a}(\S^N_+)\hookrightarrow L^{2;a}(\S^N_+)$ and of the
trace operator from $H^{1;a}(\S^N_+)$ to $L^{2}(\S^{N-1})$.
\end{proof}
We will prove two multi-component formulae, the first regarding entire profiles which
are segregated on $\R^N$, the second regarding profiles which coexist on $\R^N$.
\begin{proposition}\label{thm: ACF}
Let $v_1, v_2 \in H^{1;a}(B_R^+(x_0,0))$ be continuous functions such that
\begin{itemize}
   \item  $v_1 v_2|_{\{y=0\}} = 0$, $v_i(x_0,0) = 0$;
   \item  for every non negative $\phi \in \C_0^{\infty}(B_R(x_0,0))$,
   \[
   \int\limits_{\R^{N+1}_+}(L_a v_i)v_i\phi \, \de x \de y + \int\limits_{\R^N} (\partial_\nu^a v_i)v_i\phi \,
   \de x = \int\limits_{\R^{N+1}_+}y^a \nabla v_i\cdot\nabla(v_i\phi) \, \de x \de y \leq0.
   \]
\end{itemize}
Then the function
\[
    \Phi(r) := \prod_{i=1}^{2} \frac{1}{r^{2\nuACF}} \int\limits_{B_r^+(x_0,0)} y^a
\frac{|\nabla v_i|^2}{|X|^{N-2s} } \,\de{x} \de{y}
\]
is monotone non decreasing in $r$ for $r \in (0,R)$.
\end{proposition}
\begin{proof}
Applying the same estimates developed for the proof of Proposition \ref{prp: ACF y 2s}, it is easy to see that the proposition is equivalent to (summing equation \eqref{eqn: stima dal basso ACF} for the two functions)
\[
   \Phi'(r) \geq 0 \Leftrightarrow \sum_{i=1}^2 \ddfrac{\int\limits_{\partial^+ B_r^+} y^a \tfrac{|\nabla v_i|^2}{|X|^{N-1}}
\,\de{\sigma} }{\int\limits_{B_r^+} y^a \tfrac{|\nabla v_i|^2}{|X|^{N-1}} \,\de{x} \de{y}}
\geq \frac{2}{r} \inf_{(\omega_1, \omega_2) \in \mathcal{P}^2} \sum_{i=1}^{2}
\gamma\left(\lambda_1^s(\omega_i) \right)
= \frac{4}{r}\nuACF
\]
In particular, the last inequality follows by the definition of $\nuACF$.
\end{proof}

\begin{proposition}\label{thm: ACF perturbed}
Let $v_1,v_2 \in H^{1;a}_{\loc}\left(\overline{\R^{N+1}_+}\right)$ be continuous functions such that, for every non negative $\phi \in \C^{\infty}_0\left(\overline{\R^{N+1}_+}\right)$ and $j\neq i $,
\begin{multline*}
\int\limits_{\R^{N+1}_+}(L_a v_i)v_i\phi \, \de x \de y + \int\limits_{\R^N} ( \partial_\nu^a v_i + a_{ij}v_iv_j^2)
v_i\phi \,\de x \\
= \int\limits_{\R^{N+1}_+} y^a \nabla v_i\cdot\nabla(v_i\phi) \, \de x \de y + \int\limits_{\R^N}a_{ij}v_i^2v_j^2\phi \, \de x  \leq0.
\end{multline*}
For any $\nu' \in (0, \nuACF)$ there exists $\bar{r} >1$ such that the function
\[
    \Phi(r) := \prod_{i=1}^{2} \Phi_i(r)
\]
is monotone non decreasing in $r$ for $r \in (\bar{r}, \infty)$, where
\[
    \Phi_i(r) := \frac{1}{r^{2\nu'}}\left(\int\limits_{B_r^+} y^a |\nabla v_i|^2 {\Gamma_1}
\,\de{x} \de{y} +  \int\limits_{\partial^0 B_r^+}a_{ij} v_i^2 v_j^2 {\Gamma_1} \,\de{x}
\right), \quad \text{ for } j \neq i.
\]
\end{proposition}
The proof of Proposition \ref{thm: ACF perturbed} is based on a contradiction argument, and follows the lines of the one of Proposition \ref{thm: ACF}. We do not report the details, referring the reader to \cite[Lemma 2.5]{nttv} and
\cite[Theorem 2.13]{tvz1}, where similar computations were developed for the case $s=1$ and $s = 1/2$, respectively.

\subsection{Almgren type monotonicity formula}\label{subsec: Alm seg prof}
To conclude this section on monotonicity formulae, we focus our attention on an Almgren quotient defined for a suitable class a functions: these will come into play as limits of a blow up sequence. First, for any
\[
    \bv \in H^{1;a}_{\loc}\left(\overline{\R^{N+1}_+}\right) := \{v : \forall D \subset \R^{N+1} \text{ open and bounded}, v|_{D^+} \in H^{1;a}(D^+)\},
\]
$\bv=(v_1,\dots,v_k)$ continuous, let use define
\[
    \begin{split}
    E(x_0, r) &:= \frac{1}{r^{N-2s}} \int\limits_{B^+_r(x_0,0)} y^a \tsum_{i}
|\nabla v_i|^2 \, \de{x} \de{y},\\
    H(x_0,r) &:= \frac{1}{r^{N+1-2s}}\int\limits_{\partial^+ B^+_r(x_0,0)} y^a
\tsum_{i} v_i^2 \, \de{\sigma},
    \end{split}
\]
where $x_0 \in \R^N$ and $r > 0$. By assumption, both $E$ and $H$ are locally absolutely continuous functions on $(0,+\infty)$, that is, both $E'$ and $H'$ are $L^1_{\loc}(0,\infty)$ (here, $'=\de/\de r$). Let us also consider the function (Almgren frequency function)
\[
    N(x_0,r) := \frac{E(x_0,r)}{H(x_0, r)}.
\]
We have the following result, which proof we omit since it follows with minor changes from the one of Theorem 3.3 in \cite{tvz1}.
\begin{proposition}\label{thm:_Almgren_for_classG_s}
Let $\bv \in H^{1;a}_\loc\left(\overline{\R^{N+1}_+}; \R^k\right)$, $\bv=(v_1,\dots,v_k)$ continuous, and let us assume that:
\begin{enumerate}
 \item $v_i v_j |_{y = 0} = 0$ for every $j\neq i$;
 \item for every $i$,
\begin{equation}\label{eqn: equation of classG_s}
    \begin{cases}
    L_a v_i = 0 & \text{ in } \R^{N+1}_+\\
    v_i \partial_{\nu}^a v_i = 0 & \text{ on } \R^N \times \{0\};
    \end{cases}
\end{equation}
\item for any $x_0\in\R^N$ and a.e. $r>0$, the following (Pohozaev type) identity holds
\[
    (2s-N) \int\limits_{ B^+_r} y^a \tsum_{i} |\nabla v_i|^2 \, \de x \de y + r \int\limits_{\partial^+ B^+_r} y^a \tsum_{i} |\nabla v_i|^2 \, \de \sigma = 2 r \int\limits_{\partial^+ B^+_r} y^a \tsum_{i} |\partial_{\nu} v_i|^2 \, \de \sigma.
\]
\end{enumerate}
Then for every $x_0 \in \R^N$ the Almgren frequency function $N(x_0,r)$ is well defined on $(0,\infty)$, absolutely continuous, non decreasing, and it satisfies the identity
\begin{equation}\label{eqn: logarithmic derivative of H}
    \frac{\mathrm{d}}{\mathrm{d}r} \log H(r) = \frac{2N(r)}{r}.
\end{equation}
Moreover, if $N(r)\equiv \gamma$ on an open interval, then $N\equiv\gamma$ for every $r$, and
$\bv$ is a homogeneous function of degree $\gamma$.
\end{proposition}
Of the many consequences that the validity of an Almgren monotonicity formula carries, at this stage we are mostly interested in the following, which states a rigidity property implied by H\"older continuity.
\begin{corollary}\label{cor: holder homogeneous}
If $\bv$ satisfies the assumptions of Proposition \ref{thm:_Almgren_for_classG_s} and is globally H\"older continuous of exponent $\gamma$ on $\R^{N+1}_+$,
then it is homogeneous of degree $\gamma$ with respect to any of its (possible) zeroes, and thus
\[
\mathcal{Z}:=\{x \in \R^N: \bv(x,0) = 0\}\quad\text{is an affine subspace of }\R^N.
\]
\end{corollary}
\begin{proof}
The proof relies on the fact that the Almgren centered at any point of $\mathcal{Z}$ has to be constant and equal to $\gamma$. Indeed letting $x_0 \in \mathcal{Z}$, we argue by contradiction and suppose that $N(x_0, R) > \gamma$ for some $R$. By monotonicity of $N$ we have
\[
    \frac{\mathrm{d}}{\mathrm{d}r} \log H(r) \geq \frac{2}{r} N(x_0, R) \quad \forall r \geq R
\]
and, integrating in $(R,r)$, we find
\[
    C r^{2 N(x_0, R)} \leq H(r) \leq C r^{2\gamma},
\]
a contradiction for $r$ large enough. The same reasoning provides a contradiction in the case $N(x_0,R) < \gamma$ and $r \leq R$.
\end{proof}
\section{Liouville type results}
Relying on the previous monotonicity formulae, in this section we will prove some Liouville type theorems for solution to either equations or systems involving the operator $L_a$. As a first result, we have the following.
\begin{proposition}\label{prp: pre_liouville L_a}
Let $v\in H^{1;a}_\loc\left(\overline{\R^{N+1}_+}\right)$ be continuous and satisfy
\begin{equation*}
    \begin{cases}
    L_a v = 0 & \text{in } \R^{N+1}_+ \\
    v(x,0)=0       & \text{on } \R^{N},
    \end{cases}
\end{equation*}
and let us suppose that for some $\gamma \in [0,2s)$, $C >0$ it holds
\[
    |v(X)| \leq C(1+|X|^{\gamma})
\]
for every $X$. Then $v$ is identically zero.
\end{proposition}
\begin{proof}
We remark that $v$ satisfies the assumptions of Proposition \ref{prp: ACF y 2s} for any $R$.
For $r>0$ large enough, we choose $\eta$ non negative, smooth and radial
cut-off function supported in  $B_{2r}^+$ with $\eta = 1$ in $B_r^+$ such that
\[
    \int_{\R^{N+1}_+} y^a |\nabla \eta| \leq C r^{N+1-2s}, \quad \int_{\R^{N+1}_+} |L_a \eta| \leq C r^{N-2s}
\]
(for instance, we can take $\eta$ as a smooth approximation of the function $\frac{1}{r}( 2r - |X|)$ in $B_{2r} \setminus B_r$). Moreover, let $\Gamma_1^s$ be defined as in Definition
\ref{def:Gamma_1} (in particular, it is radial and superharmonic).
Testing the equation for $v$ with ${\Gamma_1^s} v\eta$ we obtain
\[
    \int\limits_{ B_{2r}^+ } y^a |\nabla v|^2  {\Gamma_1^s} \eta \de{x} \de{y}
    \leq \int\limits_{ B_{2r}^+ \setminus B_{r}^+} \frac{1}{2} v^2
\left[ -L_a \eta {\Gamma_1^s}  + 2 y^a \nabla \eta \cdot \nabla {\Gamma_1^s} \right]
\de{x} \de{y},
\]
where we used that $\eta$ is constant in $B_r^+$. Since
$\Gamma_1^s(X)=|X|^{2s-N}$ outside $B_1$, and $|v(X)| \leq C r^{\gamma}$ outside
a suitable $B_{\bar r}$, using the notations of Proposition \ref{prp: ACF y 2s} we infer
\[
\Phi(r) = \frac{1}{r^{4s}}\left(\int\limits_{B_r^+} y^a |\nabla v|^2 {\Gamma_1^s}
\,\de{x} \de{y} \right)\leq \frac{1}{r^{4s}} \cdot C r^{2\gamma},
\]
with $C$ independent of $r>\bar r$. Due to the monotonicity of $\Phi$, we then find
\begin{equation*}
0 \leq  \Phi(\bar{r}) \leq C r^{2(\gamma-2s)}.
\end{equation*}
for every $r>\bar r$ sufficiently large. This forces $v$ to be constant.
\end{proof}
The previous proposition allows to prove an analogous result of the classical Liouville Theorem,
which holds for $L_a$-harmonic functions.
\begin{proposition}\label{prp: liouville L_a}
Let $v$ be an entire $L_a$-harmonic function defined on $\R^{N+1}$. If there exists $\gamma < 1$ such that
\[
    |v(X)| \leq C\left(1 + |X|^{\gamma}\right),
\]
then $v|_{y=0}$ is constant. Moreover, if $\gamma < \min(2s,1)$, then $v$ is constant.
\end{proposition}
\begin{proof}
It is well known (see \cite{css}) that $L_a$-harmonic functions enjoy the mean value property ($C > 0$)
\[
    v(x,0) = \frac{C}{r^{N+a}} \int\limits_{\partial B_r(x,0)} |y|^a v \, \de \sigma
\]
and, equivalently
\[
    v(x,0) = \frac{C}{R^{N+1+a}} \int\limits_{B_R(x,0)} |y|^a v \, \de \sigma.
\]
It follows, by the growth condition, that
\[
    \begin{split}
        |v(x',0)-v(x'',0)| &\leq \frac{C}{R^{N+1+a}} \int_{B_R(x',0) \triangle B_R(x'',0) } y^a |v(x,y)| \, \de x \de y \\
            &\leq \frac{C}{R^{N+1+a}} \int_{B_R(x',0) \triangle B_R(x'',0) } y^a |X|^{\gamma} \, \de x \de y \leq C R^{\gamma-1}
    \end{split}
\]
and the first conclusion follows since $\gamma < 1$. Let us now assume $\gamma < \min(2s,1)$: since $v|_{y=0}$ is constant, we can assume $v|_{y=0} \equiv 0$ and apply Proposition \ref{prp: pre_liouville L_a}.
\end{proof}
We can obtain the analogous of the classical Liouville Theorem for $s$-harmonic functions by
applying the previous result to the even reflection through $\{y=0\}$ of their $L_a$-harmonic extensions.
\begin{corollary}\label{cor: liouville L_a}
Let $v\in H^{1;a}_\loc\left(\overline{\R^{N+1}_+}\right)$ be continuous and satisfy
\begin{equation*}
    \begin{cases}
    L_a v = 0 & \text{in } \R^{N+1}_+ \\
    \partial_\nu^a v(x,0)=0       & \text{on } \R^{N},
    \end{cases}
\end{equation*}
and let us suppose that for some $\gamma < \min(2s,1)$, $C >0$ it holds
\[
    |v(X)| \leq C(1+|X|^{\gamma})
\]
for every $X$. Then $v$ is constant.
\end{corollary}
By the way, a stronger result in the direction of the above corollary is contained in \cite[Lemma 2.7]{css}.

In the same spirit of Proposition \ref{prp: pre_liouville L_a}, we provide a result
concerning $L_a$-harmonic functions which vanish on a half space of $\R^N$.
\begin{proposition}\label{prp: liouville_boundary}
Let $v\in H^{1;a}_\loc\left(\overline{\R^{N+1}_+}\right)$ satisfy the assumptions of Proposition \ref{prp: ACF sym}. Let us suppose that for some $\gamma \in [0,s)$, $C >0$ it holds
\[
    |v(X)| \leq C(1+|X|^{\gamma})
\]
for every $X$. Then $v$ is identically zero.
\end{proposition}
\begin{proof}
Again, $v$ as above fulfills the assumptions of
Proposition \ref{prp: ACF sym}. Now, assuming that $v$ is not constant, we can argue
as in the proof of Proposition \ref{prp: liouville system} obtaining a contradiction.
\end{proof}

We proceed with a lemma regarding the decay of subsolutions to a linear equation involving $L_a$.
\begin{lemma}\label{lem: decay with perturbations}
Let $M>0$ and $\delta>0$ be fixed and let $h \in L^{\infty}(\partial^0B_1^+)$ with $\|h\|_{L^{\infty}} \leq \delta$. Any $v \in H^{1;a}(B_1^+)$ non negative solution to
\[
    \begin{cases}
    L_a v \leq 0            &\text{in $B_1^+$}\\
    \partial_{\nu}^a v \leq -M v + h  &\text{on $\partial^0 B_1^+$}
    \end{cases}
\]
verifies
\[
    \sup_{\partial^0 B_{1/2}^+} \leq \frac{1+\delta}{M} \sup_{\partial^+B^+_1} v.
\]
\end{lemma}
The proof of Lemma \ref{lem: decay with perturbations} follows by a comparison argument. In order to construct an appropriate supersolution, we need a technical lemma. Let $f \in AC(\R) \cap \C^{\infty}(\R)$ be defined as
\[
    f(x) = C\int_{-\infty}^{x} \frac{1}{(1 + t^2)^{1-a/2}} \de t,
\]
where $C$ is such that $f(+\infty) = 1$.
\begin{lemma}
There exists $c > 0$ such that
\[
    (-\Delta)^s f (x) \geq -c f(x)
\]
for any $x < 0$.
\end{lemma}
\begin{proof}
The function $f$ under consideration is increasing, smooth and such that there exist $c,C > 0$ with
\[
    \lim_{|t|\to \infty} f'(t)|t|^{2-a} = C > 0 \quad \text{and} \quad \lim_{|t|\to \infty} f''(t)|t|^{3-a} = c.
\]
The $s$-laplacian of the function $f$ is well-defined. Thanks to the extension representation of the fractional laplacian, we can consider
\[
\begin{split}
    v(x,y) &= \int_{\R} P_{a}(\xi,y) f(x-\xi) \mathrm{d}\xi = \int_{\R} y^{1-a}\frac{f(x-\xi)}{ (\xi^2+y^2)^{1-a/2} } \mathrm{d} \xi \\
    &= \left\{ t = \xi/y \right\} = \int_{\R} \frac{f(x-ty)}{(1+t^2)^{1-a/2} } \mathrm{d} t
\end{split}
\]
so that
\[
    \begin{split}
    \partial^a_{\nu} v(x,0) &= \lim_{y \to 0^+} -y^a \frac{\partial}{\partial y} \int \frac{f(x-ty)}{(1+t^2)^{1-a/2} } \mathrm{d} t = \lim_{y \to 0^+} \int y^{a}t \frac{f'(x-ty)}{(1+t^2)^{1-a/2} } \mathrm{d} t \\
                     & = \left\{ r = yt \right\} = \lim_{y \to 0^+} \int \frac{r}{(y^2+r^2)^{1-a/2} }f'(x-r) \mathrm{d} r \\
    &= \pv \int \frac{|r|^a}{r}f'(x-r) \mathrm{d} r = \pv \int \frac{|x-r|^a}{x-r} f'(r) \mathrm{d} r.
    \end{split}
\]
Let us observe that, due to the decay properties of $f'$ at infinity, the last principal value acts only around the singularity $x=r$, that is
\[
    (-\Delta)^s f(x) = \lim_{\eps \to 0^+} \int_{\R \setminus (r-\eps, r+\eps)} \frac{|x-r|^a}{x-r} f'(r) \mathrm{d} r.
\]
We aim at proving that there exists a positive $c>0$ such that the estimate
\[
    (-\Delta)^s f (x) \geq -c f(x)
\]
holds for every $x \leq 0$. As a first step, we are going to estimate the asymptotic behavior of the right hand side as $x \to -\infty$. To this end, letting $K > 0$ be a fixed number, we write
\begin{equation}\label{eqn: first decomposition}
    (-\Delta)^s f (x) =  \pv \int_{-\infty}^{-K}  \frac{|x-r|^a}{x-r} f'(r) \mathrm{d} r + \int_{-K}^{\infty}  \frac{|x-r|^a}{x-r} f'(r) \mathrm{d} r
\end{equation}
(this decomposition is possible thanks to the prescribed decay of $f'$). We estimate the two contributions separately. First $(a<1)$
\[
    \int_{-K}^{\infty}  \frac{|x-r|^a}{x-r} f'(r) \mathrm{d} r \geq -(-K-x)^{a-1} \int_{-K}^{\infty} f'(r) \mathrm{d} r \geq - C|x|^{a-1}.
\]
We further decompose the second integral in \eqref{eqn: first decomposition}, to find
\begin{multline*}
    \pv \int_{-\infty}^{-K}  \frac{|x-r|^a}{x-r} f'(r) \mathrm{d} r =\left\{ t = r/|x| \right\} =-|x|^a \pv \int_{-\infty}^{-K/|x|}  \frac{|1+t|^a}{1+t} f'(t|x|) \mathrm{d} t \\
    = -|x|^a \left[\int_{-\infty}^{-3/2}  
    \dots\, \mathrm{d} t + \pv \int_{-3/2}^{-1/2}  
    \dots\, \mathrm{d} t  + \int_{-1/2}^{-K/|x|}  
    \dots\, \mathrm{d} t\right]
\end{multline*}
In the first part we use the estimate
\[
    f'(t|x|) \geq c |t|^{a-2} |x|^{a-2}
\]
in order to obtain
\[
    -|x|^a \int_{-\infty}^{-3/2}  \frac{|1+t|^a}{1+t} f'(t|x|) \mathrm{d} t \geq - c |x|^{2a-2} \int_{-\infty}^{-3/2}  \frac{|1+t|^a}{1+t} |t|^{a-2} \mathrm{d} t \geq -C |x|^{2a-2}.
\]
In the principal value we write
\[
    -|x|^a\pv \int_{-3/2}^{-1/2}  \frac{|1+t|^a}{1+t} f'(t|x|) \mathrm{d} t = -|x|^{2a-2} \pv \int_{-3/2}^{-1/2}  \frac{|1+t|^a}{1+t} f'(t|x|)|x|^{2-a} \mathrm{d} t.
\]
Since
\[
    f'(t|x|)|x|^{2-a} \to C|t|^{a-2} \quad \text{in $\C^1\left(-\frac32,-\frac12\right)$ as $|x|\to \infty$}
\]
and
\[
    \pv \int_{-3/2}^{-1/2}  \frac{|1+t|^a}{1+t} |t|^{a-2} \mathrm{d} t =\left\{r = -1-t\right\}= \pv \int_{-1/2}^{1/2}  -\frac{|r|^a}{r} (r+1)^{a-2} \mathrm{d} r > 0,
\]
we obtain the lower bound
\[
    -|x|^a \pv \int_{-3/2}^{-1/2}  \frac{|1+t|^a}{1+t} f'(t|x|) \mathrm{d} t \geq -C |x|^{2a-2}.
\]
To estimate the last integral we use
\[
    f'(t|x|) \leq C |t|^{a-2} |x|^{a-2}
\]
to obtain
\begin{multline*}
    -|x|^a \int_{-1/2}^{-K/|x|}  \frac{|1+t|^a}{1+t} f'(t|x|) \mathrm{d} t \geq - C |x|^{2a-2} \int_{-1/2}^{-K/|x|}  \frac{|1+t|^a}{1+t} |t|^{a-2} \mathrm{d} t\\
     \geq- C |x|^{2a-2} \left(1 + \frac{1}{|x|^{a-1}} \right) \geq - C |x|^{a-1}.
\end{multline*}
As a consequence
\[
    (-\Delta)^s f (x) \geq - C\left( |x|^{a-1}+ |x|^{2a-2}\right) \geq - C |x|^{a-1}.
\]
On the other hand, by a direct estimate we have ($x\ll0$)
\[
    f(x) \leq C \frac{1}{|x|^{1-a}}
\]
which immediately yields that for $x \ll 0$ there exists $c>0$ such that
\[
    (-\Delta)^s f (x) \geq - c f(x).
\]
Due to the positivity and regularity of $f$, this estimates extends to every $x \leq 0$.
\end{proof}

We can conclude with the proof of  Lemma \ref{lem: decay with perturbations}.

\begin{proof}[Proof of Lemma \ref{lem: decay with perturbations}]
Let us first consider, for $M>0$, the scaling $x \mapsto M^{1/2s}x$ and let us introduce the function $f_M(x) := f(M^{1/2s}x)$. It follows that
\[
    (-\Delta)^s f_M (x) = M^{2s/2s} \left[(-\Delta)^s f\right] (M^{1/2s}x) \geq -cM f_M(x)
\]
It is then clear that if we let
\[
    g_M(x) := f_M(t-1) + f_M(-t-1)
\]
then for any $M>0$ it holds
\[
    \begin{cases}
    (-\Delta)^s g_M (x) \geq -cM g_M(x) &\text{ in $(-1,1)$ }\\
    g_M(x) \geq \frac12  &\text{ in $\R \setminus (-1,1)$ }\\
    g_M(x) \leq C M^{-1}  &\text{ in $\left(-\frac12,\frac12\right)$. }
    \end{cases}
\]
The proof follows by a comparison argument between $v$ and the supersolution
\[
    w_{\delta} := \delta \frac{1}{M} + \int_{\R} P_{a}(\xi,y) g_{M} (x-\xi) \mathrm{d}\xi. \qedhere
\]
\end{proof}
The previous estimate allows to prove the following.
\begin{proposition}\label{prp: global eigenfunction}
Let $v$ satisfy
\begin{equation*}
    \begin{cases}
    L_a v = 0 & \text{in } \R^{N+1}_+ \\
    \partial_{\nu}^a v = - \lambda v & \text{on } \R^{N}
    \end{cases}
\end{equation*}
for some $\lambda > 0$ and let us suppose that
for some $\gamma < \min(1,2s)$, $C >0$ it holds
\[
    |v(X)| \leq C(1+|X|^{\gamma})
\]
for every $X$. Then $v$ is constant.
\end{proposition}
\begin{proof}
Let either $z=v^+$ or $z=v^-$. In both cases,
\[
    \begin{cases}
    L_a z \leq 0, & \text{in }\R^{N+1}_+\\
    \partial_{\nu}^a z \leq -\lambda z, & \text{on }\R^N.
    \end{cases}
\]
By translating and scaling, Lemma \ref{lem: decay with perturbations} implies that
\[
z(x_0,0) \leq \sup_{\partial^0 B_{r/2}(x_0,0)} z \leq \frac{1}{\lambda r}\sup_{\partial^+ B_r(x_0,0)} z
\leq C \frac{1+r^{\gamma}}{r}.
\]
Letting $r\to\infty$ the proposition follows.
\end{proof}
\begin{proposition}\label{prp: prescribed constant normal derivative}
Let $v$ satisfy
\begin{equation*}
    \begin{cases}
    L_a v = 0 & \text{in } \R^{N+1}_+ \\
    \partial_{\nu}^a v = \lambda  & \text{on } \R^{N}
    \end{cases}
\end{equation*}
for some $\lambda \in \R$ and let us suppose that
for some $\gamma < \min(1,2s)$, $C >0$ it holds
\[
    |v(X)| \leq C(1+|X|^{\gamma})
\]
for every $X$. Then $v$ is constant.
\end{proposition}
\begin{proof}
For $h \in \R^N$, let $w(x,y) := v(x+h,y) - v(x,y)$. Then $w$ solves
\[
    \begin{cases}
    L_a w = 0 & \text{in } \R^{N+1}_+ \\
    \partial_{\nu}^a w = 0 & \text{on } \R^{N}
    \end{cases}
\]
and, as usual, we can reflect and use the growth condition to infer that $w$ has to be constant, that is $v(x+h,y) = c_h + v(x,y)$. Deriving the previous expression in $x_i$, we find that
\[
    v(x,y) = \sum_{i = 1}^k c_i(y) x_i + c_0(y).
\]
Using again the growth condition, we see that $c_i \equiv 0$ for $i = 1, \dots, k$, while $c_0$ is constant. We observe that, consequently, $\lambda = 0$.
\end{proof}
\begin{proposition}\label{prp: liouville system}
Let $\bv\in H^{1;a}_\loc\left(\overline{\R^{N+1}_+}\right)$ be continuous and satisfy
\begin{equation*}
    \begin{cases}
    L_a v_i = 0 & \text{in } \R^{N+1}_+ \\
    \partial_\nu^a v_i = - v_i\tsum_{j\neq i} a_{ij} v_j^2       & \text{on } \R^{N},
    \end{cases}
\end{equation*}
and let $\nuACF$ be defined according to \eqref{eqn: def_nuacf}.
If for some $\gamma \in (0,\nuACF)$ there exists $C$ such that
\begin{equation*}
    |\bv(X)| \leq C\left(1 + |X|^{\gamma}\right),
\end{equation*}
for every $X$, then $k-1$ components of $\bv$ annihilate and the last is constant.
\end{proposition}
\begin{proof}
We only sketch the proof, referring to \cite[Proposition 4.1]{tvz1} for a detailed proof
in the case $s=1/2$. To start with, we observe that any pair of components of $\bv$ satisfy the assumptions of Proposition \ref{thm: ACF perturbed}; as a consequence, if  $\bv$ had two nontrivial components, then one could
argue as in the proof of Proposition \ref{prp: pre_liouville L_a} in order to obtain
a contradiction. Once we know that all but one component are trivial, we can conclude by applying
Corollary \ref{cor: liouville L_a} to the last one.
\end{proof}
\begin{proposition}\label{prp: liouville inequalities}
Let $\bv$ satisfy the assumptions of Proposition \ref{thm:_Almgren_for_classG_s} and $\gamma \in (0,\nuACF)$.
\begin{enumerate}
 \item If there exists $C$ such that
\begin{equation*}
    |\bv(X)| \leq C\left(1 + |X|^{\gamma}\right),
\end{equation*}
for every $X$, then $k-1$ components of $\bv$ annihilate;
 \item if furthermore $\bv \in \C^{0,\gamma}\left(\overline{\R^{N+1}_+}\right)$ and
 \[
\gamma <
\begin{cases}
\nuACF   &  0 < s \leq \dfrac12\\
\min(\nuACF, 2s-1) & \dfrac12 < s < 1,
\end{cases}
\]
then the only possibly nontrivial component is constant.
\end{enumerate}
\end{proposition}
\begin{proof}
To prove 1., we can reason as in the proof of Proposition \ref{prp: liouville system}, using Proposition
\ref{thm: ACF} instead of Proposition \ref{thm: ACF perturbed}. Turning to 2.,
let $v$ denote the only non trivial component. If $v(x,0)\neq 0$ for every $x$, then we deduce that $\partial_\nu^a v(x,0)\equiv 0$, and we can conclude by using Corollary \ref{cor: liouville L_a}. On the other hand, let
\[
\mathcal{Z} = \{x \in \R^N: v(x,0) = 0\}\neq\emptyset.
\]
By Corollary \ref{cor: holder homogeneous}, we have that $v$ is $\gamma$-homogeneous about any point of $\mathcal{Z}$, which is then an affine subspace of $\R^N$, and that $v$ solves
\begin{equation}\label{eqn: capacity}
    \begin{cases}
    L_a v = 0 & \text{in } \R^{N+1}_+ \\
    v = 0 & \text{on } \mathcal{Z}\\
    \partial_{\nu}^a v = 0 & \text{on } \R^{N}\setminus \mathcal{Z}.
    \end{cases}
\end{equation}
Now, if $\mathcal{Z}=\R^N$, then Proposition \ref{prp: pre_liouville L_a} applies. On the other hand, if $\dim \mathcal{Z} \leq N-2s$, we obtain that $\mathcal{Z}$ has null $L_a$-capacity (this can be seen directly for the fractional laplacian in $\R^N$, see for instance \cite[Theorem 3.14]{land}), and the conclusion follows by Proposition \ref{prp: liouville L_a}. Finally, we are left to deal with the case
\[
\dim \mathcal{Z} = N-1\quad\text{and}\quad \frac12< s <1.
\]
In this situation, the previous capacitary reasoning fails, see Remark \ref{rem: capacity} below.
Nonetheless, assuming without loss of generality that $\mathcal{Z} = \{x \in \R^N: x_1 = 0\}$, we have
that $v$ satisfies the assumptions of Proposition \ref{prp: ACF N-1}. As a consequence, one can reason once again
as in the proof of Proposition \ref{prp: pre_liouville L_a}, obtaining a contradiction with the fact
that $\gamma<2s-1$.
\end{proof}
\begin{remark}\label{rem: capacity}
As we already mentioned in the introduction, in great contrast with the case $s\leq 1/2$, if $s>1/2$ the fundamental solution of the $s$-laplacian in $\R$ is bounded in a neighborhood of $x=0$. As a consequence, the function $\Gamma(x,y) = |(x_1,y)|^{2s-1}$ solves \eqref{eqn: capacity}. This implies that, for $s>1/2$, the sets of codimension 1 in $\R^N$ have positive $s$-capacity.
\end{remark}
%
\section{$\C^{0,\alpha}$ uniform bounds}\label{section:_uniform_local}
%
%
%
In this section we turn to the proof of the regularity results we stated in the introduction.
In particular we will prove Theorem \ref{thm:_local_holder}. We recall that, here and in the following, the functions $f_{i,\beta}$
appearing in problem $\problem{\beta}^s$ are assumed to be continuous and uniformly bounded, with respect to $\beta$, on bounded sets. We start by recalling the regularity results which hold for $\beta$ bounded. For easier notation, we write $B^+=B^+_{1}$.
\begin{lemma}\label{lem: sire_dupaigne}
There exists $\alpha^*\in(0,1)$ such that, for every $\alpha\in(0,\alpha^*)$, $\bar m>0$ and $\bar\beta>0$,
there exists a constant $C=C(\alpha,\bar m,\bar\beta)$ such that
\[
    \| \bv_\beta\|_{\C^{0,\alpha}\left(\overline{B^+_{1/2}}\right)} \leq C,
\]
for every $\bv_\beta$ solution of problem $\problem{\beta}$ on $B^+$, satisfying
\[
\beta\leq\bar\beta\quad\text{ and }\quad  \| \bv_{\beta} \|_{L^{\infty}(B^+)} \leq \bar m.
\]
\end{lemma}
\begin{proof}
The above regularity issue can be rephrased for a general $h\in H^{1;a}(B^+)$ with
\[
\begin{cases}
L_a h = 0   & \text{in }B^+\\
h = f \in L^\infty      & \text{on }\partial^+B^+\\
\partial_\nu^a h = g \in L^\infty &\text{on }\partial^0B^+.
\end{cases}
\]
Denoting
\[
\tilde f(x,y) := f(x,|y|) \quad\text{and}\quad \tilde g(x)=
\begin{cases}
g(x) & x\in \partial^0 B^+\\
0    & x \in \R^N\setminus\partial^0 B^+,
\end{cases}
\]
we can write $h = h_1 + h_2$, where
\[
\begin{cases}
L_a h_1 = 0   & \text{in }\R^{N+1}_+\\
\partial_\nu^a h_1 = \tilde g  &\text{on }\R^N
\end{cases}
\quad\text{and}\quad
\begin{cases}
L_a h_2 = 0   & \text{in }B\\
h_2 = \tilde f - h_1       & \text{on }\partial B.
\end{cases}
\]
But then the regularity of $h_1$ (depending on $\|\tilde g\|_{L^\infty}$) follows by
\cite[Proposition 2.9]{silve}, while the one of $h_2$ is proved in \cite{fks} (see
also \cite[Section 2]{css}).
\end{proof}
From now on, without loss of generality, we will fix $\alpha^*>0$ in such a way that Lemma
\ref{lem: sire_dupaigne} holds, and furthermore
\[
\alpha^* \leq
\begin{cases}
\nuACF   &  0 < s \leq \dfrac12\\
\min(\nuACF, 2s-1) & \dfrac12 < s < 1.
\end{cases}
\]
We will obtain Theorem \ref{thm:_local_holder} for any fixed $\alpha \in (0,\alpha^*)$. Following the outline of \cite[Section 6]{tvz1}, we proceed by
contradiction and develop a blow up analysis. Let $\eta$ denote a smooth function such that
\begin{equation*}
 \begin{cases}
 \eta(X) = 1   &  0\leq |X|\leq \frac12\\
 0<\eta(X) \leq 1   &  \frac12\leq |X|\leq 1\\
 \eta(X) = 0   &  |X|=1\\
 \end{cases}
\end{equation*}
(in particular, $\eta$ vanishes on $\partial^+B^+$ but is strictly positive $\partial^0B^+$).
We will show that
\[
\| \eta\bv\|_{\C^{0,\alpha}\left(\overline{B^+}\right)} \leq C,
\]
and the theorem will follow by the definition of $\eta$.
Let us assume by contradiction the existence of sequences
$\{\beta_n\}_{n \in \N}$, $\{\mathbf{v}_n\}_{n\in \N}$,
solutions to $\problem{\beta_n}^s$, such that
\[
    L_n := \max_{i = 1, \dots, k} \max_{X'\neq X'' \in \overline{B^+}} \frac{|(\eta v_{i,n})(X')-(\eta v_{i,n})(X'')|}{|X'-X''|^{\alpha}} \rightarrow \infty.
\]
By Lemma
\ref{lem: sire_dupaigne} (and the regularity of $\eta$) we infer that
$\beta_n \rightarrow \infty$. Moreover, up to a relabeling, we may assume that $L_n$ is achieved
by $i = 1$ and by two sequences of points $(X'_n, X''_n) \in \overline{B^+} \times
\overline{B^+}$. The first properties of such sequences have already been obtained in \cite{tvz1}.
\begin{lemma}[\cite{tvz1}, Lemma 6.4]\label{lem: acc non a part+}
Let $X'_n \neq X''_n$ and $r_n := |X'_n-X''_n|$ satisfy
\begin{equation*}
    L_n = \frac{|(\eta v_{1,n})(X'_n)-(\eta v_{1,n})(X''_n)|}{r_n^{\alpha}}.
\end{equation*}
Then, as $n\to\infty$,
\begin{enumerate}
 \item $r_n \rightarrow 0$;
 \item $\dfrac{\dist(X'_n,\partial^+ B^+)}{r_n}\to \infty$,
 $\dfrac{\dist(X''_n,\partial^+ B^+)}{r_n}\to \infty$.
\end{enumerate}
\end{lemma}
Our analysis is based on two different blow up sequences, one having uniformly bounded
H\"older quotient, the other satisfying a suitable problem. Let $\{{\hat{X}_n}\}_{n\in\N}\subset
\overline{B^+}$, $|{\hat{X}_n}|<1$, be a sequence of points, to be chosen later. We write
\[
    \tau_n B^+ := \frac{B^+ - {\hat{X}_n}}{r_n},
\]
remarking that $\tau_n B^+$ is a hemisphere, not necessarily centered on the hyperplane
 $\{y=0\}$. We introduce the sequences
\[
    w_{i,n}(X) := \eta({\hat{X}_n}) \frac{v_{i,n}({\hat{X}_n} + r_n X)}{L_n r_n^{\alpha}} \quad \text{and} \quad \bar{w}_{i,n}(X) := \frac{(\eta v_{i,n})({\hat{X}_n} + r_n X)}{L_n r_n^{\alpha}},
\]
where $X \in \tau_n B^{+}$. With this choice, on one hand it follows immediately that, for every $i$ and
$X'\neq X'' \in \overline{\tau_n B^{+}}$,
\begin{align*}
   \frac{|\bar{w}_{i,n}(X')-\bar{w}_{i,n}(X'')|}{
    |X'-X''|^{\alpha}} \leq & \left|\bar{w}_{1,n}\left(\frac{X'_n-{\hat{X}_n}}{r_n}\right)-
    \bar{w}_{1,n}\left(\frac{X''_n-{\hat{X}_n}}{r_n}\right)\right| = 1,
\end{align*}
in such a way that the functions $\{\bar{\bw}_n\}_{n \in \N}$ share an uniform bound on
H\"older seminorm, and at least their first components are not constant.
On the other hand, since $\eta({\hat{X}_n})>0$,
each $\bw_{n}$ solves
\begin{equation}\label{eqn: w_sol}
    \begin{cases}
    L_a^{\tau_n} w_{i,n} = 0 & \text{ in }\tau_nB^{+}\\
    \partial_{\nu}^{a,\tau_n} w_{i,n} = f_{i,n}(w_{i,n}) - M_n w_{i,n} \tsum_{j \neq i}
    a_{ij} w_{j,n}^2 & \text{ on } \tau_n \partial^0 B^+,
    \end{cases}
\end{equation}
where the new operators write ($\hat X_n = (\hat x_n, \hat y_n)$)
\[
L_a^{\tau_n}  = -\div\left(\left(\hat{y}_{{n}}r_n^{-1}+y\right)^a\nabla \right),
\qquad
\partial_{\nu}^{a,\tau_n} = \lim_{y \to (-\hat{y}_{{n}}r_n^{-1})^+} -
\left(\hat{y}_{{n}}r_n^{-1}+y\right)^a \partial_y,
\]
and $f_{i,n}(t) = \eta({\hat{X}_n}) r_n^{2s-\alpha} L_n^{-1} f_{i,\beta_n}(L_n
r_n^{\alpha} t/\eta({\hat{X}_n}))$, $M_n = \beta_n L_n^2 r_n^{2\alpha + 2s}/\eta({\hat{X}_n})^2$.
\begin{remark}\label{rem:f_i to 0}
The uniform bound of $\| \bv_{\beta} \|_{L^{\infty}}$
imply that
\[
\sup_{\tau_n \partial^0 B^+}|f_{i,n}(w_{i,n})| = \eta({\hat{X}_n}) r_n^{2s-\alpha} L_n^{-1}
\sup_{\partial^0 B^+}|f_{i,\beta_n}
\left(v_{i,n}\right)|\leq C(\bar m)r_n^{2s-\alpha} L_n^{-1}\to 0
\]
as $n\to\infty$.
\end{remark}
A crucial property is that the two blow up sequences defined above have
asymptotically equivalent behavior, as enlighten in the following lemma.
\begin{lemma}[\cite{tvz1}, Lemma 6.6]\label{lem:_uniform_convergence_w_and_bar_w}
Let $K\subset\R^{N+1}$ be compact. Then
\begin{enumerate}
  \item  $\displaystyle\max_{X \in K\cap\overline{\tau_n B^+}} | \bw_{n}(X)- \bar{\bw}_{n}(X)| \to 0$;
   \item there exists $C$, only depending on $K$, such that $|\bw_{n} (X)- \bw_{n}(0)| \leq C$, for every $x\in K$.
\end{enumerate}
\end{lemma}
Now we show that the sequences $(X'_n,X''_n)$ accumulates towards $\{y=0\}$.
\begin{lemma}\label{lem: shift up_local}
There exists $C>0$ such that, for every $n$ sufficiently large,
\[
\frac{\dist(X'_n, \partial^0 B^+) + \dist(X''_n, \partial^0 B^+)}{r_n} \leq C.
\]
\end{lemma}
\begin{proof}
We argue by contradiction. Taking into account the second part of Lemma \ref{lem: acc non a part+},
this forces
\[
    \frac{\dist(X'_n, \partial B^+) + \dist(X''_n, \partial B^+)}{r_n} \to \infty.
\]
In the definition of $\bw_n$, $\bar\bw_n$ we choose $\hat X_n = X'_n$, so that $\tau_n B^+ \to \R^{N+1}$ and
$\hat y^{-1}_n r_n\to 0$. Let $K$ be any fixed compact set. Then, by definition, $K$ is contained in the half sphere $\tau_n B^+$, for every $n$ sufficiently large. By defining $\bW_{n} = \bw_{n} - \bw_{n}(0)$,  $\bar\bW_{n} = \bar\bw_{n} - \bar\bw_{n}(0)$, we obtain that $\{\bar\bW_n\}_{n\in\N}$ is a sequence of functions which share the same $\C^{0,\alpha}$-seminorm and are uniformly bounded in $K$, since $\bar\bW_n(0) = 0$. By the Ascoli-Arzel\`a Theorem, there exists a function $\bW \in C(K)$ which, up to a subsequence, is the uniform limit of $\{\bar\bW_n\}_{n\in\N}$: taking a countable compact exhaustion of $\R^{N+1}$ we find that $\bar\bW_n \to \bW$ uniformly in every compact set. Moreover, for any pair $X$, $Y$, we have that $X, Y \in \tau_n B^+$ for every $n$ sufficiently large, and so
\[
    |\bar \bW_{n}(X) - \bar\bW_{n}(Y)| \leq \sqrt{k} |X-Y|^\alpha.
\]
Passing to the limit in $n$ the previous expression, we obtain $\bW \in \C^{0,\alpha}(\R^{N+1})$. By Lemma \ref{lem:_uniform_convergence_w_and_bar_w}, we also find that $\bW_n \to \bW$ uniformly on compact sets.
We want to show that $\bW$ is harmonic. To this purpose, let $\varphi \in \C_0^{\infty}(\R^{N+1})$ be a smooth test function, and let $\bar n$ be sufficiently large so that $\supp \varphi \subset \tau_n B^+$ for all $n \geq \bar n$. For a fixed $i \in \{1,\dots,k\}$, we test the equation $L_a^{\tau_n} w_{i,n} = 0$ by $\varphi$ to find
\[
    \int_{\R^{N+1}} -\div\left(\left(1+yr_n\hat{y}_n^{-1}\right)^a\nabla \varphi \right) w_{i,n} \, \de x \de y = 0.
\]
Passing to the uniform limit and observing that $(1+yr_n\hat{y}_n^{-1})^a\to 1$ in $\C^{\infty}(\supp\varphi)$, we obtain at once that $\bW$ is indeed harmonic. We will obtain a contradiction with the classical Liouville Theorem once we show that $\bW$ is not constant. To this aim we observe that $(X'_n-\hat{X}_n)/r_n = 0$ and, up to a subsequence,
\[
    \frac{X''_n-{\hat{X}_n}}{r_n} = \frac{X''_n-X'_n}{|X''_n-X'_n|}\to X''\in \partial{B_{1}}.
\]
Therefore, by equicontinuity and uniform convergence,
\[
    \left|\bar W_{1,n}\left(\frac{X'_n-{\hat{X}_n}}{r_n}\right) - \bar W_{1,n}\left(\frac{X''_n-{\hat{X}_n}}{r_n}\right)\right| = 1 \implies |W_{1}(0) - W_{1}(X'')| = 1. \qedhere
\]
\end{proof}
After the result above, we are in a position to choose ${\hat{X}_n}$ in the definition of $\bw_n$, $\bar\bw_n$
as
\[
{\hat{X}_n}:=(x'_n,0),
\]
where as usual $X'_n=(x'_n,y'_n)$. With this choice, it is immediate to see that
\[
 L_a^{\tau_n} =  L_a,\quad
 \partial_{\nu}^{a,\tau_n} = \partial_{\nu}^{a},\quad
 \tau_nB^+\to\Omega_\infty = \R^{N+1}_+.
\]
Moreover, by Lemma \ref{lem: shift up_local}, we have that $X'_n,X''_n\in B^+_C$, for some $C$ not depending on $n$.
This will imply that any possible blow up limit can not be constant. Now one can reason as in \cite[Section
6]{tvz1} in order to prove that the blow up sequences converge. In doing this, a first crucial step
consists in proving that $\bw_n(0)$ is bounded: to this aim, it is useful to notice that the decay rate
for subsolutions which we obtained in Lemma \ref{lem: decay with perturbations} does not depend on $s$ and completely agrees with the one found in \cite[Lemma 4.5]{tvz1}. Consequently, the uniform bound on the H\"older seminorm
allows to prove the
following result.
\begin{lemma}[\cite{tvz1}, Lemma 6.13]\label{lem: uniform implies strong convergence local}
Under the previous blow up setting, there exists $\bw \in
(H^{1;a}_{\loc}\cap \C^{0,\alpha})\left(\overline{\R^{N+1}_+}\right)$ such that, up to a subsequence,
\[
	\bw_{n} \to \bw \text{ in }(H^1\cap C)(K)
\]
for every compact $K\subset\overline{\R^{N+1}_+}$
.
\end{lemma}
\begin{proof}[End of the proof of Theorem \ref{thm:_local_holder}]
Up to now, we have that $\bw_{n} \to \bw$ in $(H^{1;a}\cap C)_{\loc}$,
and that the limiting blow up profile $\bw$ is a nonconstant vector
of harmonic, globally H\"older continuous functions. To reach the final
contradiction, we distinguish, up to subsequences, between
the following three cases.

\textbf{Case 1:} $M_n\to0$. In this case also the equation on the
boundary passes to the limit, and the nonconstant component
$w_1$ satisfies $\partial_{\nu}^a w_{1} \equiv 0$ on $\R^{N}$, in
contradiction with Corollary \ref{cor: liouville L_a}.

\textbf{Case 2:} $M_n\to C>0$. Even in this case the equation on the boundary passes to the limit, and
$\bw$ solves
\[
    \begin{cases}
    L_a w_{i} = 0 & x \in \R^{N+1}_+\\
    \partial_{\nu}^a w_{i} = - C w_i \tsum_{j \neq i} a_{ij} w_j^2 & \text{on } \R^{N}
\times \{0\}
    \end{cases}
\]
The contradiction is now reached using Proposition \ref{prp: liouville
system}.

\textbf{Case 3:} $M_n\to \infty$. In this case we can find a contradiction with
Proposition \ref{prp: liouville inequalities}. To this aim, one has to prove the validity of a
Pohozaev-type identity for the limits of the blow-up sequences. This can be done by taking into account
Lemma \ref{lem: uniform implies strong convergence local} and reasoning as in \cite[Section 5]{tvz1}.

As of now, the contradictions we have obtained imply that $\{\bv_\beta\}_{\beta>0}$ is
uniformly bounded in $\C^{0,\alpha}\left(\overline{B^+_{1/2}}\right)$, for every $\alpha<\alpha^*$. But
then the relative compactness in $\C^{0,\alpha}\left(\overline{B^+_{1/2}}\right)$ follows by
Ascoli-Arzel\`a Theorem, while the one in $H^{1;a}({B^+_{1/2}})$ can be shown by reasoning as
in the proof of Lemma \ref{lem: uniform implies strong convergence local}.
\end{proof}

Incidentally, we remark that similar arguments can be exploited in order to prove the following compactness
result, concerning segregated profiles (see also \cite[Proposition 6.15]{tvz1}).
This result, though technical at this stage,
provides a compactness criterion for suitable blow down sequences, and may be useful in proving optimal
regularity results, along the scheme explained in the introduction.
\begin{proposition}
Let $\{\bv_n\}_{n \in \N}$ be a subset of $\C^{0,\alpha}\left(\overline{B^+_1}\right)$,
for some $0<\alpha\leq\alpha^*$, and satisfy the assumptions of Proposition \ref{thm:_Almgren_for_classG_s}.
If
\[
    \| \bv_{n} \|_{L^{\infty}(B^+_1)} \leq \bar m,
\]
with $\bar m$ independent of $n$, then for every $\alpha' \in (0,\alpha)$ there exists a constant
$C = C(\bar m,\alpha')$, not depending on $n$, such that
\[
    \| \bv_n\|_{\C^{0,\alpha'}\left(\overline{B^+_{1/2}}\right)} \leq C.
\]
Furthermore, $\{\bv_n\}_{n\in\N}$ is relatively compact in $H^{1;a}(B^+_{1/2}) \cap \C^{0,\alpha'}\left(\overline{B^+_{1/2}}\right)$ for every $\alpha' < \alpha$.
\end{proposition}
To conclude, we mention that the above local result can be used, together with a covering argument and
Proposition \ref{prp: liouville_boundary}, to prove Theorem \ref{thm: glob_bdd_intro} (see also
\cite[Theorem 8.5]{tvz1}): there are, however, two different situations to be handled.

First, if one considers the problem \eqref{eqn: general_system} set on the whole $\R^N$ (Theorem \ref{thm: glob_bdd_intro} in the case $\Omega = \R^N$), then the global uniform bounds on $\bu_{\beta}$ imply, by the representation formula of Caffarelli and Silvestre \cite{cs}, that also $\bv_{\beta}$ enjoy the same uniform $L^{\infty}$ bounds. As a consequence, the local uniform bounds extend at once to the global case by a simple covering argument.

In the case of $\Omega \neq \R^N$, one has to deal also with the boundary of $\Omega$. In this situation, the regularity for $\bu_{\beta}$ is ensured by \cite{serra}, while the uniform H\"older bounds - obtained again via the blow up analysis - follows with similar arguments and the use of the appropriate Liouville type results (Proposition \ref{prp: liouville_boundary}). Further details can be found in \cite[Section 8]{tvz1}.


\end{document}